\documentclass{amsart}
\usepackage{a4wide}
\usepackage[latin1]{inputenc}
\usepackage[T1]{fontenc}
\usepackage{amsmath,amssymb,amsthm,stmaryrd}
\usepackage{mathrsfs}
\usepackage{graphicx}
\usepackage{esint}

\newcommand{\ud}[0]{\,\mathrm{d}}

\newcommand{\essinf}[0]{\operatornamewithlimits{ess\,inf}}

\newcommand{\diverg}[0]{\operatorname{div}}
\newcommand{\dist}[0]{\operatorname{dist}}
\newcommand{\proj}[0]{\mathbb{P}}

\newcommand{\eps}[0]{\varepsilon}

\newcommand{\rGamma}[0]{\reflectbox{\(\Gamma\)}}

\newcommand{\abs}[1]{|#1|}
\newcommand{\Babs}[1]{\Big|#1\Big|}

\newcommand{\Norm}[2]{\|#1\|_{#2}}
\newcommand{\BNorm}[2]{\Big\|#1\Big\|_{#2}}
\newcommand{\pair}[2]{\langle #1,#2 \rangle}
\newcommand{\Bpair}[2]{\Big\langle #1,#2 \Big\rangle}

\newcommand{\bddlin}[0]{\mathscr{L}}
\newcommand{\kernel}[0]{\mathsf{N}}
\newcommand{\range}[0]{\mathsf{R}}
\newcommand{\domain}[0]{\mathsf{D}}

\renewcommand{\Re}[0]{\operatorname{Re}}

\newcommand\R{\mathbf{R}}
\newcommand\C{\mathbf{C}}
\newcommand\N{\mathbf{N}}

\newcommand{\Exp}[0]{\mathbb{E}}

\newcommand{\radem}[0]{\varepsilon}

\newcommand{\Rad}[0]{\operatorname{Rad}}

\swapnumbers

\numberwithin{equation}{section}

\makeatletter
  \let\c@subsection\c@equation
\makeatother

\theoremstyle{plain}
\newtheorem{theorem}[equation]{Theorem}
\newtheorem{proposition}[equation]{Proposition}
\newtheorem{corollary}[equation]{Corollary}
\newtheorem{lemma}[equation]{Lemma}

\theoremstyle{definition}

\theoremstyle{remark}
\newtheorem{remark}[equation]{Remark}
\newtheorem{example}[equation]{Example}

\makeatletter
\@namedef{subjclassname@2010}{
  \textup{2010} Mathematics Subject Classification}
\makeatother

\title[Stability of the $H^{\infty}$-calculus]{Stability in $p$ of the $H^{\infty}$-calculus \\ of first-order systems in $L^p$}

\author[Hyt\"onen]{Tuomas Hyt\"onen}
\address{Department of Mathematics and Statistics, University of Helsinki, Gustaf H\"allstr\"omin katu 2b, FI-00014 Helsinki, Finland}
\email{tuomas.hytonen@helsinki.fi}

\author[McIntosh]{Alan McIntosh} 
\address{Centre for Mathematics and its Applications, Australian National University, Canberra ACT 0200, Australia}
\email{Alan.McIntosh@anu.edu.au}

\date{\today}


\begin{document}

\begin{abstract}
We study certain differential operators of the form $AD$ arising from a first-order approach to the Kato square root problem. We show that if such operators are $R$-bisectorial in $L^p$, they remain $R$-bisectorial in $L^q$ for all $q$ close to $p$. In combination with our earlier results with Portal, which required such $R$-bisectoriality in different $L^q$ spaces to start with,  this shows that the $R$-bisectoriality in just one $L^p$ actually  implies bounded $H^{\infty}$-calculus in $L^q$ for all $q$ close to $p$. We adapt the approach to related second-order results developed by Auscher, Hofmann and Martell, and also employ abstract extrapolation theorems due to Kalton and Mitrea.
\end{abstract}

\maketitle

\section{Introduction}

Recall that an operator $\mathcal{A}$ on a Banach space $X$ is called \emph{bisectorial} of angle $\omega\in[0,\pi/2)$ if its spectrum satisfies
\begin{equation*}
  \sigma(\mathcal{A})\subseteq S_{\omega}:=\Sigma_{\omega}\cup(-\Sigma_{\omega}),\qquad
  \Sigma_{\omega}:=\{z\in\C;\abs{\arg(z)}\leq\omega\},
\end{equation*}
and there holds
\begin{equation*}
  \Norm{(I+\tau\mathcal{A})^{-1}}{\bddlin(X)}\leq C_{\omega'}\qquad\forall\ \tau\notin S_{\omega'},\quad\forall\ \omega'>\omega.
\end{equation*}
For such an operator, one can define a calculus of bounded operators by formal substitution to the Cauchy integral formula,
\begin{equation*}
\begin{split}
  \psi(\mathcal{A}) &:=\frac{1}{2\pi i}\int_{\partial S_{\omega'}}\psi(\lambda)(I-\frac{1}{\lambda}\mathcal{A})^{-1}\frac{\ud\lambda}{\lambda}, \\
  \psi &\in H^{\infty}_0(S_{\omega''}):=\{\phi\in H^{\infty}(S_{\omega''}):\phi\in O\big((\frac{z}{1+z^2})^{\alpha}\big),\alpha>0\},\qquad\omega''>\omega'>\omega,
\end{split}
\end{equation*}
and it is of interest whether this calculus may be boundedly extended to all $\psi\in H^{\infty}(S_{\omega''})$ (bounded holomorphic functions in the interior of $S_{\omega''}$). If this is the case for all $\omega''>\omega$, then $\mathcal{A}$ is said to have a \emph{bounded $H^{\infty}$-calculus} of angle $\omega$.

A bisectorial operator $\mathcal{A}$ is called \emph{$R$-bisectorial} of angle $\omega$ if all sequences of operators $T_k$ taken from the (bounded) collection of resolvents $(I+\tau\mathcal{A})^{-1}$, $\tau\notin S_{\omega'}$, for any $\omega'>\omega$, satisfy the stronger \emph{$R$-boundedness} condition
\begin{equation*}
  \Exp\BNorm{\sum_k\radem_k T_k u_k}{X}\leq C  \Exp\BNorm{\sum_k\radem_k u_k}{X},
\end{equation*}
where the $\radem_k$ are random signs and $\Exp$ is the corresponding expectation. This is a condition of boundedness on the space $\Rad X$ of all sequences $(u_k)_k\subset X$ for which the series $\sum_k\radem_k u_k$ converges almost surely and equipped with the norm on the right of the previous displayed line.

In our two papers with Portal \cite{HMP1,HMP2}, we have used this  $R$-bisectoriality estimate to characterize the boundedness of the $H^{\infty}$-calculus of certain first-order differential operators, which are related to a first-order approach to the Kato square root problem for second-order divergence-form operators as developed by Axelsson, Keith and McIntosh \cite{AKM} and streamlined in \cite{AAM}, where further applications of this formalism to boundary value problems are also described. The main results of \cite{HMP1,HMP2} state that the considered operators
have a bounded $H^{\infty}$-calculus of angle $\omega$ in $X=L^p(\R^n;\C^N)$ for all $p$ in an open interval, if and only if they are $R$-bisectorial of angle $\omega$ in $L^p$ for all $p$ in the same interval. Notice that in this situation $\Rad X\eqsim L^p(\R^n;\Rad\C^N)\eqsim L^p(\R^n;\ell^2(\C^N))$, so that the abstract $R$-bisectoriality condition reduces to a classical-style square function estimate.

The most immediate deficiency of the mentioned results in \cite{HMP1,HMP2} is that they only work for an open interval of exponents, rather than a fixed one. The aim of this paper is to remove this deficiency by showing that, in fact, the $R$-bisectoriality of the specific operators of interest already self-improves from one $L^p$ to an open interval of $L^q$ spaces, thus making the previous results applicable with the a priori weaker assumption. In doing so, we follow the line of investigation of extrapolating $L^p$ inequalities, which was started by Blunck and Kunstmann \cite{BluKu} (with some prehistory going back to Duong and Robinson \cite{DuoRo}) and elaborated in the context of operators related to Kato's problem by Auscher, Hofmann and Martell \cite{Auscher,HofMa}. However, it seems that the existing extrapolation results by these authors have always dealt with the related second-order operators only. There are also more recent first-order results by Ajiev~\cite{Ajiev}, but the scopes of his and our assumptions and conclusions are not immediately comparable.

Although we are able to adapt the main lines of the approach of \cite{HofMa}, with an intermediate application of a result from \cite{Auscher}, to our situation, this was not completely obvious. One key difference of the first-order operators compared to the second-order counterparts is the existence of non-trivial null-spaces which have to be dealt with. While, in principle, the strategy just consists of separating the treatment on the complementary subspaces of the range and the kernel, finding the right places in the proof to make the splitting in a technically correct way required some trial and error.

We also adopt the generality of the related first-order papers \cite{Ajiev,AAM,AKM,HMP2} by considering coefficient matrices $A$, which are only required to satisfy a coercivity condition $\Norm{Au}{p}\gtrsim\Norm{u}{p}$ in the $L^p$ sense and only on the (in general not dense) range of a relevant differential operator $D$. This is in contrast to the usual uniform ellipticity assumptions made in the second-order treatments like \cite{Auscher,HofMa}. We have included a comparison of the different coercivity conditions in an appendix where we also show, for all constant-coefficient matrices $A$, that the $L^p$ coercivity on the range of $D$ is equivalent to the uniform pointwise coercivity on the range of the Fourier multiplier symbol of~$D$. For a general $A$, however, our necessary and sufficient pointwise conditions do not meet. See \cite[Section~2]{AAMarkiv} and \cite[Section~1]{AHMT} for a related discussion of $L^2$ and pointwise accretivity conditions for second and higher order divergence-form operators.

Let us finally note that, for quite a while, we also had in mind another potential strategy towards eliminating the need of assumptions concerning an interval of values of $p$. In fact, the proofs of \cite{HMP1,HMP2} contain just one particular obstruction against the possibility of working with a fixed $p$, namely, an $L^p$ version of Carleson's embedding theorem. In contrast to the classical $L^2$ result, it needs an assumption for $p+\eps$ to get a conclusion for the given $p>2$. Our alternative hope was to eliminate this $\eps$ from the $L^p$ Carleson inequality but we now know that this cannot be achieved, at least not on the level of the general embedding result. A counterexample for $p=4$ was first constructed and shown to one of us by Michael Lacey (personal communication, September 2009), and it can be extended to all $p>2$. The details are presented in a more general context elsewhere \cite{HytKe}.

In the following section, we give a precise formulation of our extrapolation result sketched above, which is then proven in the rest of the paper. The coercivity conditions, as mentioned, are further discussed in an appendix.

\section{Set-up and main results}\label{sec:setup}

We work in the Lebesgue spaces $L^q:=L^q(\R^n;\C^N)$ with $q\in(1,\infty)$. 

\subsection{The operator $D$}
We denote by $D$ a first order constant-coefficient differential operator
\begin{equation*}\tag{D0}\label{eq:D0}
  D=-i\sum_{j=1}^n\hat{D}_j\partial_j,\qquad\hat{D}_j\in\bddlin(\C^N),
\end{equation*}
acting on $\C^N$-valued Schwartz distributions. It can also be viewed as the Fourier multiplier operator with symbol $\hat{D}(\xi)=\sum_{j=1}^n\hat{D}_j\xi_j$. This induces an unbounded operator on each $L^q$ with domain $\domain_q(D):=\{u\in L^q;Du\in L^q\}$. The symbol is required to satisfy the following properties:
\begin{equation*}\tag{D1}\label{eq:D1}
  \kappa\abs{\xi}\abs{e}\leq\abs{\hat{D}(\xi)e}\qquad\forall\ \xi\in\R^n,\quad\forall\ e\in\range(\hat{D}(\xi)),
\end{equation*}
where $\range(\hat{D}(\xi))$ stands for the range of $\hat{D}(\xi)$, and
\begin{equation*}\tag{D2}\label{eq:D2}
  \sigma(\hat{D}(\xi))\subseteq S_{\omega}:=\Sigma_{\omega}\cup(-\Sigma_{\omega}),\qquad
  \Sigma_{\omega}:=\{z\in\C;\abs{\arg(z)}\leq\omega\},
\end{equation*}
where $\kappa>0$ and $\omega\in[0,\pi/2)$ are some constants. 

Under these assumptions, it has been shown in \cite[Lemma~4.1]{HMP2} that
\begin{equation}\label{eq:equivCondD}
  \sigma(\hat{D}(\xi))\subseteq [S_{\omega}\cap A(\kappa\abs{\xi},M\abs{\xi})]\cup\{0\},\qquad
  \C^N=\kernel(\hat{D}(\xi))\oplus\range(\hat{D}(\xi)),
\end{equation}
where $A(a,b):=\{z\in\C;\ a\leq\abs{z}\leq b\}$ and $M:=\sup_{\abs{\xi}=1}\abs{\hat{D}(\xi)}<\infty$, and $\kernel(\hat{D}(\xi))$ stands for the kernel of $\hat{D}(\xi)$. This condition, conversely, implies the original assumptions on $\hat{D}(\xi)$ (possibly with a different $\kappa$). In particular, \eqref{eq:equivCondD} gives that the spectrum of $\hat{D}(\xi)$ restricted to its range satisfies $\sigma\big(\hat{D}(\xi)|\range(\hat{D}(\xi))\big)\subseteq A(\kappa\abs{\xi},M\abs{\xi})$, and then, by Cramer's rule, that the inverse of this restricted operator has norm bounded by $C\abs{\xi}^{-1}$, which is a reformulation of \eqref{eq:D1}.  Moreover, the condition \eqref{eq:equivCondD} is equivalent to the corresponding statement for $\hat{D}(\xi)^*$, and hence everything we say about $D$ is also true for $D^*$. This fact was implicitly used in \cite{HMP2}.

It was proven in \cite[Theorem~5.1]{HMP2} that $D$ is bisectorial and has a bounded $H^{\infty}$-calculus of angle $\omega$ in $L^q$ for all $q\in(1,\infty)$. Consequently, there is a direct sum decomposition $L^q=\kernel_q(D)\oplus\overline{\range_q(D)}$ into the kernel $\kernel_q(D):=\{u\in\domain_q(D);Du=0\}$ and the closure of the range $\range_q(D)=\{Du;u\in\domain_q(D)\}$. The two components are complemented in $L^p$ with the common projections $\proj_D^0$ and $\proj_D^1$, where for instance the former can be represented by
\begin{equation*}
  \proj_D^0 u=\lim_{\tau\to\infty}(I+\tau D)^{-1}u,
\end{equation*}
where the limit is along $\tau\notin S_{\omega'}$ with $\omega'>\omega$. Hence, the kernels $\kernel_q(D)$ form an interpolation scale for $q\in(1,\infty)$, and the same is true for the spaces $\overline{\range_q(D)}$.

Moreover \cite[Proposition~5.2]{HMP2}, $D$ satisfies the property, for all $q\in(1,\infty)$, that
\begin{equation*}
  \Norm{\nabla u}{q}\lesssim\Norm{Du}{q}\qquad\forall\ u\in\domain_q(D)\cap\overline{\range_q(D)}\subseteq W^{1,q}.
\end{equation*}

\subsection{The operators $A$ and $AD$}
Let
\begin{equation*}\tag{A}\label{eq:A}
  A\in L^{\infty}(\R^n;\bddlin(\C^N))
\end{equation*}
be a bounded matrix-valued function, frequently identified with the pointwise multiplication operator acting boundedly on all $L^q$. Its adjoint $A^*$ is an operator of the same type.

Our primary interest is in the composition of the two operators just defined. Many operators of interest in partial differential equations arise in this form. In particular, with $A=\begin{pmatrix} I & 0 \\ 0 & A_1 \end{pmatrix}$ and $D=\begin{pmatrix} 0 & -\diverg \\ \nabla & 0 \end{pmatrix}$, we have $(AD)^2=\begin{pmatrix} L & 0 \\ 0 & \tilde{L} \end{pmatrix}$, where $L$ is the second-order divergence form operator $L=-\diverg A_1\nabla$. Proving the boundedness of the $H^{\infty}$-calculus of $AD$ in $L^p$ implies the Kato square root estimate $\Norm{\sqrt{L}u}{p}\eqsim\Norm{\nabla u}{p}$; see \cite[Sec.~2]{AAM} and Corollary~\ref{cor:Kato} below.

An important property of the operators $AD$ is that their resolvents $(I+\tau AD)^{-1}$, as soon as bounded on $L^p$, automatically satisfy the following localized bounds, often called off-diagonal estimates. (The result is stated in \cite[Proposition~5.1]{AAM} for $p=2$, but the same proof works for any $p\in(1,\infty)$.)

\begin{lemma} [Off-diagonal estimates; \cite{AAM}, Prop.~5.1]
Let $A$ and $D$ be as in \eqref{eq:A}, \eqref{eq:D0}, \eqref{eq:D1} and \eqref{eq:D2}. There is an $\alpha>0$ with the following property. Suppose that $\Norm{(I+\tau AD)^{-1}}{p\to p}\lesssim 1$. If $E$ and $F$ are disjoint Borel sets and $u\in L^p$ is supported on $F$, then
\begin{equation*}
  \Norm{1_E(I+\tau AD)^{-1}u}{p}\lesssim e^{-\alpha\dist(E,F)/\abs{\tau}}\Norm{u}{p}
    \lesssim\Big(\frac{\dist(E,F)}{\abs{\tau}}\Big)^{-K}\Norm{u}{p}
\end{equation*}
for any $K>0$.
\end{lemma}

We are now ready for the formulation of the main theorem.

\begin{theorem}\label{thm:extrapRbisec}
Let $A$ and $D$ be as in \eqref{eq:A}, \eqref{eq:D0}, \eqref{eq:D1} and \eqref{eq:D2}. 
Suppose that the following conditions hold for some $p\in(1,\infty)$:
the operator $AD$ is $R$-bisectorial of angle $\omega$ in $L^p$,
the operator $A^*D^*$ is $R$-bisectorial of angle $\omega$ in $L^{p'}$,
and we have the coercivity estimates
\begin{equation}\label{eq:coercivity}
  \Norm{Au}{p}\gtrsim\Norm{u}{p}\quad\forall u\in\range_p(D),\qquad
  \Norm{A^*v}{p'}\gtrsim\Norm{v}{p'}\quad\forall v\in\range_{p'}(D^*).
\end{equation}
Then these conditions remain valid with $p$ replaced by any $q$ in some open interval containing $p$. Hence $AD$ has a bounded $H^{\infty}$-calculus of angle $\omega$ in $L^q$ for all these $q$, in particular for $q=p$.
\end{theorem}

\begin{remark}
(i) If $A$ is invertible on $L^p$ (and then $A^*$ on $L^{p'}$), then the $R$-bisectoriality of $AD$ in $L^p$ implies the $R$-bisectoriality of $A^*D^*=A^*(D^*A^*)(A^*)^{-1}=A^*(AD)^*(A^*)^{-1}$ in $L^{p'}$ by duality and similarity, and so the conditions on the dual operators can be removed.

(ii) More generally, the weaker assumption
\begin{equation*}
  L^{p'}=\kernel_{p'}(A^*D^*)\oplus\overline{\range_{p'}(A^*D^*)}
\end{equation*}
also allows us to remove the conditions on the dual operators because it implies that $A^*D^*$ on $\overline{\range_{p'}(A^*D^*)}$ is similar to $D^*A^*=(AD)^*$ on $\overline{\range_{p'}(D^*A^*)}$. In fact, writing $\tilde{A}^{-*}:\overline{\range_{p'}(A^*D^*)}\to\overline{\range_{p'}(D^*)}=\overline{\range_{p'}(D^*A^*)}$ for the inverse of $A^*:\overline{\range_{p'}(D^*)}\to\overline{\range_{p'}(A^*D^*)}$, we have that  $A^*D^*=A^*(D^*A^*)\tilde{A}^{-*}$ on $\overline{\range_{p'}(A^*D^*)}$. This restricted similarity suffices, since the resolvent bounds of $A^*D^*$ on $\kernel_{p'}(A^*D^*)$ are trivial.

(iii) Only the claim starting with ``Then'' requires proof in Theorem~\ref{thm:extrapRbisec}; the claim starting with ``Hence'' then follows from \cite[Corollary~8.17]{HMP2}. In fact, the mentioned Corollary is stated for operators of the form $DA$ rather than $AD$. But our assumptions are symmetric in $AD$ and $A^*D^*=(DA)^*$, so once we have proven the $R$-bisectoriality of $AD$ in $L^q$ and thus, by symmetry, of $A^*D^*$ in $L^{q'}$, we also have the $R$-bisectoriality of $D^*A^*$ in $L^{q'}$ and $DA$ in $L^q$ by duality. Then \cite[Corollary~8.17]{HMP2} applies directly to give the bounded $H^{\infty}$-calculus of $D^*A^*$ and $DA$, and we get back to $AD$ and $A^*D^*$ by duality again.
\end{remark}

As a matter of fact, one can somewhat weaken the $R$-bisectoriality assumptions in Theorem~\ref{thm:extrapRbisec}:

\begin{theorem}\label{thm:weakTypeRbisec}
Let $A$ and $D$ be as in \eqref{eq:A}, \eqref{eq:D0}, \eqref{eq:D1} and \eqref{eq:D2}. 
For some $p\in(1,\infty)$, let $AD$ be bisectorial of angle $\omega$ in $L^p$, let $A^*D^*$ be bisectorial of angle $\omega$ in $L^{p'}$, and assume the coercivity \eqref{eq:coercivity} and the following weak-type $R$-bisectoriality inequalities:
\begin{equation*}
  \Big\{x\in\R^n;\Exp\Babs{\sum_k\radem_k(I+\tau_kAD)^{-1}u_k}>\alpha\Big\}
  \leq \frac{C}{\alpha^p}\int_{\R^n}\Exp\Babs{\sum_k\radem_k u_k}^p\ud x,
\end{equation*}
and let further $A^*D^*$ have similar bounds in the dual space $L^{p'}$:
\begin{equation*}
  \Big\{x\in\R^n;\Exp\Babs{\sum_k\radem_k(I+\tau_kA^*D^*)^{-1}v_k}>\alpha\Big\}
  \leq \frac{C}{\alpha^{p'}}\int_{\R^n}\Exp\Babs{\sum_k\radem_k v_k}^{p'}\ud x,
\end{equation*}
both uniformly for all $\tau_k\notin S_{\omega'}$ ($\omega'>\omega$) and $\alpha>0$. Then the conclusions of Theorem~\ref{thm:extrapRbisec} still hold.
\end{theorem}

We note that the majority of the results in \cite{HMP2}, and all the results in \cite{HMP1}, are actually formulated somewhat differently from the $AD$ (or $DA$) formalism of \cite{AAM} employed here, treating instead operators of the form $\Gamma+B_1\rGamma B_2$ (with $\Gamma$ and $\rGamma$ differentiation, $B_1$ and $B_2$ multiplication operators) introduced in~\cite{AKM}. However, one can usually transfer results back and forth between the two frameworks (cf. \cite[Section 10.1]{AAM} and the proofs of \cite[Corollaries 8.17, 9.3]{HMP2}), and it now seems that the $AD$ operators are conceptually simpler and at least equally useful.

Here is a consequence for the Kato square root problem for systems \cite{AHMT}:

\begin{corollary}\label{cor:Kato}
Let $A_1\in L^{\infty}(\R^n;\bddlin(\C^m\otimes\C^n))$ satisfy
\begin{equation}\label{eq:L2coer}
  \int_{\R^n}\nabla\bar{u}(x)\cdot A_1(x)\nabla u(x)\ud x\gtrsim\Norm{\nabla u}{2}^2,
\end{equation}
for all $u\in W^{1,2}(\R^n;\C^m)$. Then $AD$, with $\displaystyle A=\begin{pmatrix} I & 0 \\ 0 & A_1 \end{pmatrix}$ and $\displaystyle D=\begin{pmatrix} 0 & -\diverg \\ \nabla & 0 
\end{pmatrix}$, has a bounded $H^{\infty}(S_{\omega})$-calculus in $L^p(\R^n;\C^m\oplus[\C^m\otimes\C^n])$ for all $p\in(p_0,p_1)$, for some $p_0<2<p_1$ and $\omega\in(0,\pi/2)$. In particular, $L=-\diverg A_1\nabla$ satisfies $\Norm{\sqrt{L}u}{p}\eqsim\Norm{\nabla u}{p}$ for all $u\in W^{1,p}(\R^n;\C^m)$ and $p\in(p_0,p_1)$.
\end{corollary}

\begin{proof}
It is immediate that \eqref{eq:L2coer} implies \eqref{eq:coercivity} with $p=2$. It has been shown in \cite{AAM} (cf.~\cite{AKM}) that these conditions imply the bisectoriality (and in fact the bounded $H^{\infty}(S_{\omega})$-calculus) of $AD$ and $A^*D^*$ in $L^2(\R^n;\C^m\oplus[\C^m\otimes\C^n])$. In a Hilbert space, bisectoriality coincides with $R$-bisectoriality. Hence all the assumptions of Theorem~\ref{thm:extrapRbisec} are verified, and the mentioned theorem implies the asserted conclusion.
\end{proof}

The estimate $\Norm{\sqrt{L}u}{p}\eqsim\Norm{\nabla u}{p}$ for $p$ in a neighbourhood of $2$ (with more precise information on the values $p_0$ and $p_1$) was shown by Auscher \cite{Auscher} in the case of scalar-valued functions, i.e., $m=1$ in Corollary~\ref{cor:Kato}.

\section{Beginning of the proof}

\subsection{Preparation for the proof}
Interestingly, the specific form of the inequality that we want to extrapolate plays very little role in the proof. What matters is that we are concerned with the boundedness of some operators
\begin{equation*}
  T= (T_k)_k:u=(u_k)_k\mapsto(T_k u_k)_k
\end{equation*}
on $L^p(\R^n;\Rad(\C^N))$, or from $L^p(\R^n;\Rad(\C^N))$ to $L^{p,\infty}(\R^n;\Rad(\C^N))$, where the components $T_k$ are in the $H^{\infty}$-calculus of $AD$. In fact, $T_k=(I+\tau_k AD)^{-1}$ in the situation at hand.

The inequality assumed in Theorem~\ref{thm:weakTypeRbisec} says that
\begin{equation*}
  \abs{\{x\in\R^n;\abs{Tu}>\alpha\}}\leq\frac{C}{\alpha^p}\Norm{u}{p}^p,
\end{equation*}
where we write simply $\abs{\ }$ for the norm of $\Rad\C^N$ and $\Norm{\ }{p}$ for that of $L^p(\R^n;\Rad\C^N)$. It will suffice to prove a similar weak-type inequality for all $q$ in an open neighbourhood of $p$, for then the asserted strong-type inequalities in the same range follow from Marcinkiewicz' interpolation theorem.

Note that we are not assuming the resolvents $(I+\tau_k AD)^{-1}$ to act a priori boundedly on $L^q$, and so the expression $Tu$ need not be well-defined for all $u\in L^q(\R^n;\Rad\C^N)$. Of course, we will first consider $u$ in a dense subspace consisting of functions in $L^p\cap L^q$, but the choice of the subspace now needs slightly more care than in the usual Calder\'on--Zygmund theory, and we return to this issue in a moment.

\subsection{Abstract operator extrapolation}
The first coercivity condition in \eqref{eq:coercivity} says that the mapping $A:\overline{\range_p(D)}\to L^p$ is bounded from below. By an extrapolation result of Kalton and Mitrea \cite[Theorem~2.5]{KalMit}, using that both $\overline{\range_q(D)}$ and $L^q$ form interpolation scales, it remains bounded from below for all $q$ in some open interval $(p_0,p_1)$ containing $p$, i.e.,
\begin{equation*}
  \Norm{Au}{q}\gtrsim\Norm{u}{q}\quad\forall u\in\overline{\range_q(D)},\quad\forall q\in(p_0,p_1).
\end{equation*}
This in turn implies that $\kernel_q(AD)=\kernel_q(D)$ and $\overline{\range_q(AD)}=A\overline{\range_q(D)}$, and so even the spaces $\kernel_q(AD)$ (by equality to $\kernel_q(D)$) and $\overline{\range_q(AD)}$ (by isomorphism to $\overline{\range_q(D)}$) form interpolation scales for $q\in(p_0,p_1)$.

The assumed bisectoriality implies the topological direct sum splitting $L^p=\kernel_p(AD)\oplus\overline{\range_p(AD)}$, with the associated projections denoted by $\proj^0_{AD}$ and $\proj^1_{AD}$. (Recall also the corresponding splitting $L^p=\kernel_p(D)\oplus\overline{\range_p(D)}$, with projections $\proj^0_{D}$ and $\proj^1_{D}$.) An equivalent formulation of this topological splitting is the isomorphism of the mapping
\begin{equation*}
  J_p:\kernel_p(AD)\oplus\overline{\range_p(AD)}\to L^p,\quad
  (u_0,u_1)\mapsto u_0+u_1.
\end{equation*}
The similarly defined mapping $J_q$ is obviously bounded for every $q\in(p_0,p_1)$ and of course it coincides with $J_p$ on the intersection of their domains. Since the involved spaces form interpolation scales again, another extrapolation result (see \cite[Theorem~2.7]{KalMit}; the particular case needed here actually goes back to {\v{S}}ne{\u\i}berg \cite{Sneiberg}) shows that $J_q$ remain an isomorphism for all $q$ in a possibly smaller open interval containing $p$. By adjusting the numbers $p_0$ and $p_1$ if necessary, we keep denoting this interval by $(p_0,p_1)$. Thus
\begin{equation*}
  L^q=\kernel_q(AD)\oplus\overline{\range_q(AD)},\quad\forall q\in(p_0,p_1).
\end{equation*}

\subsection{Splitting the proof into kernel and range}
We first concentrate on the exponents $q\in(p_0,p)$, and want to establish the weak-type $R$-bisectoriality estimate there. The topological splitting just established allows the separate treatment of $u_k\in\kernel_q(AD)$ and $u_k\in\overline{\range_q(AD)}$. We now choose appropriate subspaces of these two spaces, where the operators of interest are well-defined, and can be eventually extended by density to all $L^q$.

Note that $\kernel_q(AD)\cap L^p\subseteq\kernel_p(AD)$ and similarly with $q$ and $p$ interchanged; hence
\begin{equation*}
  L^q\cap L^p = [\kernel_q(AD)\cap\kernel_p(AD)]\oplus[\overline{\range_q(AD)}\cap\overline{\range_p(AD)}].
\end{equation*}
Since this space is dense in $L^q$, the two components on the right are dense in $\kernel_q(AD)$ and $\overline{\range_q(AD)}$.

For the kernel, we simply take the subspace
\begin{equation*}
   \kernel_q(AD)\cap\kernel_p(AD)\subseteq\kernel_q(AD),
\end{equation*}
and observe that the estimate of interest is a triviality there, since $(I+\tau_k AD)^{-1}u_k=u_k$ for $u_k\in\kernel_p(AD)$.

By definition, $\range_q(AD)$ is dense in $\overline{\range_q(AD)}$. Elements of this space are of the form $ADf$, where $f\in\domain_q(D)$, and replacing $f$ by $\proj_D^1f$, we may assume that $f\in\domain_q(D)\cap\overline{\range_q(D)}\subseteq W^{1,q}$. Let us then approximate $f$ in the $W^{1,q}$ norm by an element $\tilde{f}\in W^{1,q}\cap W^{1,p}\subseteq\domain_q(D)\cap\domain_p(D)$. Then $AD\tilde{f}\in\range_q(AD)\cap\range_p(AD)$ is close to $AD f$ in the $L^q$ norm. The key estimate will then be proven for the functions
\begin{equation*}
  u_k=ADf_k\in \range_q(AD)\cap\range_p(AD)\subseteq \overline{\range_q(AD)},
\end{equation*}
where we can further assume that
\begin{equation*}
  f_k\in\domain_q(D)\cap\overline{\range_q(D)}\cap\domain_p(D)\cap\overline{\range_p(D)}\subseteq
  W^{1,q}\cap W^{1,p},
\end{equation*}
since the spaces $\overline{\range_q(D)}$ are complemented in the respective $L^q$ by the common projection $\proj^1_D$.

\subsection{Calder\'on--Zygmund decomposition}

We make use of the Calder\'on--Zygmund decomposition for Sobolev functions due to Auscher \cite[Lemma 5.12]{Auscher} and then follow the procedure of Blunck and Kunstmann~\cite{BluKu}, or perhaps more precisely its variant in Hofmann and Martell~\cite{HofMa}. For $\alpha>0$ and $f\in \dot{W}^{1,q}(\R^n;\Rad\C^N)$, as we have, Auscher's result provides a representation
\begin{equation*}
  f=g+\sum_j b^j,
\end{equation*}
where
\begin{equation*}
  \Norm{\nabla g}{\infty}\lesssim\alpha,\qquad  b^j\in W^{1,q}_0(Q_j;\Rad\C^N),\qquad\fint_{Q_j}\abs{\nabla b^j}^q\ud x\lesssim\alpha^q
\end{equation*}
and the $Q_j$ are cubes with
\begin{equation*}
  \sum_j\abs{Q_j}\lesssim\alpha^{-q}\int_{\R^n}\abs{\nabla f}^q\ud x,\qquad\sum_j 1_{Q_j}\lesssim 1.
\end{equation*}

As a consequence of these estimates, it follows that (using the bounded overlap of the cubes $Q_j$ in the first step)
\begin{equation*}
 \BNorm{\sum_j\nabla b^j}{q}\lesssim\Big(\sum_j\Norm{\nabla b^j}{q}^q\Big)^{1/q}\lesssim\Norm{\nabla f}{q}
  \lesssim\Norm{D f}{q}\lesssim\Norm{AD f}{q}=\Norm{u}{q};
\end{equation*}
hence also $\displaystyle\Norm{\nabla g}{q}=\BNorm{\nabla f-\sum_j\nabla b^j}{q}\lesssim\Norm{u}{q}$,
and in combination with the $L^{\infty}$ bound for $\nabla g$ we have
\begin{equation*}
 \Norm{\nabla g}{p}^p\lesssim\alpha^{p-q}\Norm{u}{q}^q,\qquad
 \qquad p\in(q,\infty).
\end{equation*}

The decomposition of $f$ immediately leads to a decomposition of $u=(u_k)_k=(ADf_k)_k=ADf$,
\begin{equation*}
  u=ADg+ADb=ADg+\sum_j ADb^j.
\end{equation*}
Then (recalling the abbreviation $T=((I+\tau_k DA)^{-1})_k$)
\begin{equation*}
  \abs{\{\abs{Tu}>3\alpha\}}
  \leq\abs{\{\abs{TADg}>\alpha\}}+\abs{\{\abs{TADb}>2\alpha\}}
\end{equation*}
As usual, the good part is estimated by the boundedness properties already known to us in $L^p$:
\begin{equation*}
  \abs{\{\abs{TADg}>\alpha\}}
  \lesssim\frac{1}{\alpha^p}\Norm{ADg}{p}^p
  \lesssim\frac{1}{\alpha^p}\Norm{\nabla g}{p}^p
  \lesssim\frac{1}{\alpha^p}\alpha^{p-q}\Norm{u}{q}^q
  =\frac{1}{\alpha^q}\Norm{u}{q}^q.
\end{equation*}

\section{Analysis of the bad part}

We turn to the estimation of the bad part, where the Blunck--Kunstmann procedure~\cite{BluKu} deviates from the classical Calder\'on--Zygmund theory. The idea of the following further decomposition goes back to Duong and Robinson~\cite{DuoRo}:
\begin{equation*}
  TAD b =\sum_j TAD b^j
  =\sum_j T(I-S_{\ell(Q_j)})ADb^j+\sum_j TS_{\ell(Q_j)}ADb^j,
\end{equation*}
where $S_t$ is  an approximation of the identity adapted to the operator $AD$. See also~\cite{DuoMc}.

To ensure a high degree of approximation, which plays a role in certain estimates below, we follow~\cite{BluKu} to introduce the auxiliary function
\begin{equation*}
  \varphi(z):=\sum_{m=0}^M\binom{M}{m}(-1)^m(1+imz)^{-1}\in H^{\infty}(S_{\omega'})\quad\forall\omega'<\pi/2.
\end{equation*}
This satisfies $\abs{\varphi(z)}\lesssim\min\{\abs{z}^M,1\}$ for all $z\in S_{\omega'}$, where the uniform bound is clear and the decay at zero follows from
\begin{equation*}
  \varphi^{(k)}(0)
  =k!\sum_{m=0}^M\binom{M}{m}(-1)^m(-im)^k
  =k!\big(-iz\frac{d}{dz}\big)^k(1-z)^M\Big|_{z=1}=0\qquad\forall k=0,\ldots,M-1.
\end{equation*}
Then we define $S_t$ by
\begin{equation*}
  I-S_t:=\varphi(tAD)=\sum_{m=0}^M\binom{M}{m}(-1)^m(I+itmAD)^{-1}.
\end{equation*}

Blunck and Kunstmann \cite{BluKu} formulated an abstract version of such higher order approximate identities, and applied it to questions of $H^{\infty}$-calculus with $e^{-tmL}$ in place of $(I+itmAD)^{-1}$ above. These semigroup-based mollifiers were also exploited by Auscher, Hofmann and Martell \cite{Auscher,HofMa}; variants involving the resolvent, as here, appear in Ajiev~\cite{Ajiev}.

Let further $E^*:=\R^n\setminus\bigcup_j 2Q_j$. Then
\begin{equation*}
\begin{split}
  \abs{\{\abs{TADb}>2\alpha\}}
  \leq\sum_j\abs{2Q_j}&+\Babs{\Big\{x\in E^*;\Babs{\sum_j T(I-S_{\ell(Q_j)})ADb^j}>\alpha\Big\}} \\
   & +\Babs{\Big\{x\in\R^n;\Babs{\sum_j TS_{\ell(Q_j)}ADb^j}>\alpha\Big\}},
\end{split}
\end{equation*}
and
\begin{equation*}
  \sum_j\abs{2Q_j}\lesssim\alpha^{-q}\Norm{\nabla f}{q}^q\lesssim\alpha^{-q}\Norm{u}{q}^q
\end{equation*}
by the properties of Auscher's Calder\'on--Zygmund decomposition, so we are left with estimating the size of the two remaining sets.

In their treatment, we will follow the approach of Hofmann and Martell \cite{HofMa}. Its perhaps most distinctive difference compared to, say, \cite{Auscher} is pushing the operators $T$, $S_t$ and $AD$ to the dual side, which effectively decouples the use of the assumptions on these operators from the use of the properties of the Calder\'on--Zygmund decomposition. This conceptual simplification was helpful to us for getting the details of the proof correctly organized.

\subsection{The mollified term with $TS_{t}$}
As $S_{\ell(Q_j)}$ is a linear combination of the resolvents $(I+im\ell(Q_j)AD)^{-1}$, $m=1,\ldots,M$, it suffices to consider just one of them. Then, using the assumed weak-type $L^p$-inequality for $T$,
\begin{equation*}
  \Babs{\Big\{\Babs{\sum_j T(I+im\ell(Q_j)AD)^{-1}AD b^j}>\alpha\Big\}}
  \leq\frac{1}{\alpha^p}\BNorm{\sum_j (I+im\ell(Q_j)AD)^{-1}AD b^j}{p}^p.
\end{equation*}
We estimate this expression by dualising with an $h\in L^{p'}(\R^n;\Rad(\C^N))$. Let further
\begin{equation*}
  S(j,0):=2Q_j,\qquad S(j,r):= 2^{r+1}Q_j\setminus 2^r Q_j,\quad r=1,2,\ldots,
\end{equation*}
and $h_{j,r}:=1_{S(j,r)}h$. Then
\begin{equation*}
\begin{split}
  &\Babs{\Bpair{\sum_j (I+im\ell(Q_j)AD)^{-1}AD b^j}{h}}
  \leq\sum_j\sum_{r=0}^{\infty}\abs{\pair{(I+im\ell(Q_j)AD)^{-1}AD b^j}{h_{j,r}}} \\
  &=\sum_j\sum_{r=0}^{\infty}\abs{\pair{b^j}{(AD)^*(I+im\ell(Q_j)(AD)^*)^{-1}h_{j,r}}} \\
  &=:\sum_j\sum_{r=0}^{\infty}\abs{\pair{b^j}{\tilde{h}_{j,r}}} 
     \leq\sum_j\sum_{r=0}^{\infty}\Norm{b^j}{p}\Norm{1_{Q_j}\tilde{h}_{j,r}}{p'},
\end{split}
\end{equation*}
where the last step used the fact that $b^j$ is supported on $Q_j$.

Next, recalling that $b^j\in W^{1,q}_0(Q_j;\Rad\C^N)$,
\begin{equation*}
  \Norm{b^j}{p}\lesssim\abs{Q_j}^{1/p-1/q+1/n}\Norm{\nabla b^j}{q}\lesssim\alpha\abs{Q_j}^{1/p+1/n}
\end{equation*}
by Sobolev's inequality and properties of the Calder\'on--Zygmund decomposition. (Here it is required that $1/p-1/q+1/n\geq 0$. If necessary, we replace the lower end-point $p_0$ of our considered interval $(p_0,p)\owns q$ by $\displaystyle\max\{p_0,\frac{pn}{p+n}\}$ to ensure this.) On the other hand,
\begin{equation*}
\begin{split}
  \Norm{1_{Q_j}\tilde{h}_{j,r}}{p'}
  &=\frac{1}{m\ell(Q_j)}\Norm{1_{Q_j}[I-(I+im\ell(Q_j)(AD)^*)^{-1}]h_{j,r}}{p'} \\
  &\lesssim\frac{1}{\ell(Q_j)}\Big(\frac{\ell(Q_j)}{2^r\ell(Q_j)}\Big)^{-K}\Norm{h_{j,r}}{p'} \\
  &\lesssim\abs{Q_j}^{-1/n}2^{-rK}(2^{rn}\abs{Q_j})^{1/p'}\Big(\fint_{2^{r+1}Q_j}\abs{h}^{p'}\ud x\Big)^{1/p'}
\end{split}
\end{equation*}
by the off-diagonal estimates satisfied by the resolvents $(I+\tau AD)^{-1}$ in $L^p$, which are easily seen to dualise.

Substituting back and taking $K>n/p'$, it follows that
\begin{equation*}
\begin{split}
  \sum_j\sum_{r=0}^{\infty}\abs{\pair{b^j}{\tilde{h}_{j,r}}} 
  &\lesssim\sum_j\sum_{r=0}^{\infty}\alpha\abs{Q_j}2^{r(n/p'-K)}\Big(\fint_{2^{r+1}Q_j}\abs{h}^{p'}\ud x\Big)^{1/p'} \\
  &\lesssim\sum_j\alpha\abs{Q_j}\essinf_{x\in Q_j} (M\abs{h}^{p'})^{1/p'}(x) \\
  &\lesssim\alpha\int_{\bigcup_j Q_j}(M\abs{h}^{p'})^{1/p'}\ud x\lesssim\alpha\Babs{\bigcup_j Q_j}^{1/p}\Norm{\abs{h}^{p'}}{1}^{1/p'}
\end{split}
\end{equation*}
by Kolmogorov's lemma (see e.g.~\cite{Meyer}, Ch.~VII, Lemme~10) and the weak-type $(1,1)$ inequality for the Hardy--Littlewood maximal operator in the last step. Taking the supremum over $\Norm{h}{p'}=\Norm{\abs{h}^{p'}}{1}^{1/p'}\leq 1$, and recalling the size of the Calder\'on--Zygmund cubes $Q_j$, it has been shown that
\begin{equation*}
  \Babs{\Big\{x\in\R^n;\Babs{\sum_j TS_{\ell(Q_j)}ADb^j}>\alpha\Big\}}
  \lesssim\frac{1}{\alpha^p}\Big(\alpha\Babs{\bigcup_j Q_j}^{1/p}\Big)^p\lesssim\alpha^{-q}\Norm{u}{q}^q.
\end{equation*}

\subsection{The remaining term with $T(I-S_t)$}
It remains to estimate
\begin{equation*}
\begin{split}
  \Babs{\Big\{x\in E^*;\Babs{\sum_j T(I-S_{\ell(Q_j)})ADb^j}>\alpha\Big\}}^{1/p}
  &\leq\frac{1}{\alpha}\BNorm{1_{E^*}\sum_j T(I-S_{\ell(Q_j)})ADb^j}{p} \\
  &=\sup\frac{1}{\alpha}\Babs{\Bpair{\sum_j T(I-S_{\ell(Q_j)})ADb^j}{h}},
\end{split}
\end{equation*}
where the supremum is over all $h\in L^{p'}(E^*;\Rad\C^N)$ with $\Norm{h}{p'}\leq 1$. As before, the pairing can be written as
\begin{equation*}
  \sum_j\sum_{r=1}^{\infty}\pair{b_j}{1_{Q_j}\tilde{h}_{j,r}},\qquad
  \tilde{h}_{j,r}:=(AD)^*(I-S_{\ell(Q_j)}^*)T^*h_{j,r},
\end{equation*}
where $h_{j,r}=1_{S(j,r)}h$ has the same meaning as earlier. Notice, however, that the summation can now begin from $r=1$, since $h_{j,0}=1_{2Q_j}h=0$ by the restriction of the support of $h$ on $E^*$ only.

Estimating $b_j$ as before, this leads to
\begin{equation*}
\begin{split}
  \sum_j\sum_{r=1}^{\infty}\abs{\pair{b_j}{\tilde{h}_{j,r}}}
  &\leq\sum_j\sum_{r=1}^{\infty}\Norm{b_j}{p}\Norm{1_{Q_j}\tilde{h}_{j,r}}{p'} \\
  &\lesssim\sum_j\sum_{r=1}^{\infty}\alpha\abs{Q_j}^{1/p+1/n}\times
     \Norm{1_{Q_j}(AD)^*(I-S_{\ell(Q_j)}^*)T^*h_{j,r}}{p'}.
\end{split}
\end{equation*}
The operators $(AD)^*(I-S_{\ell(Q_j)}^*)T^*$ (or their components; recall that we are working on sequence-valued functions) are in the $H^{\infty}$-calculus of $(AD)^*$,
\begin{equation*}
  (AD)^*(I-S_{\ell(Q_j)}^*)T^*_k
  =\frac{1}{2\pi i}\int_{\partial S_{\omega'}} z\varphi(\ell(Q_j)z)(1+\tau_k z)^{-1}(I-\frac{1}{z}(AD)^*)^{-1}\frac{\ud z}{z}.
\end{equation*}
The resolvents $(I-z^{-1}(AD)^*)^{-1}$ satisfy off-diagonal estimates on $L^{p'}$ by the assumed bisectoriality and duality, so it straightforwardly follows (now using the bound $\abs{\varphi(z)}\lesssim\max(\abs{z}^M,1)$ and taking $M>K>1$) by estimating the integral in the two parts $\abs{z}\leq\ell(Q_j)^{-1}$ and $\abs{z}>\ell(Q_j)^{-1}$, that
\begin{equation*}
   \Norm{1_{Q_j}(AD)^*(I-S_{\ell(Q_j)}^*)T^*h_{j,r}}{p'}
   \lesssim\ell(Q_j)^{-1}2^{-rK}\Norm{h_{j,r}}{p'}.
\end{equation*}
This is exactly of the same form as in the previous part of the estimate, and so we conclude just like there.

\section{Conclusion of the proof}

\subsection{Conclusion of the lower extrapolation}
We have shown that
\begin{equation*}
    \Big\{x\in\R^n;\Exp\Babs{\sum_k\radem_k(I+\tau_kAD)^{-1}u_k}>\alpha\Big\}
  \leq \frac{C}{\alpha^q}\int_{\R^n}\Exp\Babs{\sum_k\radem_k u_k}^q\ud x
\end{equation*}
for all $q\in(p_0,p)$, for some $p_0<p$. Interpolating this weak-type inequality at two different points, we deduce that
\begin{equation*}
  \Exp\BNorm{\sum_k\radem_k(I+\tau_kAD)^{-1}u_k}{q}\lesssim\Exp\BNorm{\sum_k\radem_k u_k}{q},\qquad q\in(p_0,p);
\end{equation*}
thus $AD$ is also $R$-bisectorial in $L^q$ for all these $q$.

\subsection{Upper extrapolation}
We turn to the question of $R$-bisectoriality of $AD$ for some $q>p$. First observe that the operator $A^*D^*$ in $L^{p'}$ satisfies assumptions exactly like those verified by $AD$ in $L^p$. Our lower extrapolation results imply that $A^*D^*$ is $R$-bisectorial in $L^{q'}$ for all $q'\in(p_1',p')$, with some $p_1'<p'$. We would like to show from this, that $D^*A^*$  is $R$-bisectorial in $L^{q'}$ for we would then have by duality, that $AD$ is $R$-bisectorial in $L^{q}$ for all $q\in(p, p_1)$.

By duality from the decomposition related to $AD$, we have
\begin{equation*}
  L^{q'}=\kernel_{q'}(D^*A^*)\oplus\overline{\range_{q'}(D^*A^*)}
\end{equation*}
for $q'\in(p_1',p')$ (possibly adjusting $p_1$), and the resolvent bounds on $\kernel_{q'}(D^*A^*)$ are trivial, so we need to show that 
$D^*A^*$  is $R$-bisectorial in $\overline{\range_{q'}(D^*A^*)}$. Now
$\overline{\range_{q'}(D^*A^*)}= \overline{\range_{q'}(D^*)}$ (as follows from $\kernel_q(AD)=\kernel_q(D)$ by duality) and 
$A^*\overline{\range_{q'}(D^*)}=\overline{\range_{q'}(A^*D^*)}$ (possibly readjusting $p_1$) because
\begin{equation*}
  \Norm{A^*v}{q'}\gtrsim\Norm{v}{q'}\quad\forall v\in\overline{\range_{q'}(D^*)}
\end{equation*} for all $q'\in(p_1',p')$
as follows from 
Kalton--Mitrea extrapolation as before.
That is, $A^*\overline{\range_{q'}(D^*A^*)}=\overline{\range_{q'}(A^*D^*)}$, and so the operator $D^*A^*$ on $\overline{\range_q(D^*A^*)}$ is similar to $A^*D^*$ on $\overline{\range_q(A^*D^*)}$, and hence inherits the same resolvent estimates.

Thus $D^*A^*$ is $R$-bisectorial in $L^{q'}$ for $q'\in(p_1',p')$, and by duality $AD$ is $R$-bisectorial in $L^q$ for $q\in(p,p_1)$.

Interpolating this with the estimate for $q\in(p_0,p)$, we finally obtain $R$-bisectoriality in the original space $L^p$, too.

\appendix

\section{Remarks on the coercivity condition}

In this appendix, we give some necessary and some (other) sufficient conditions for the validity of coercivity inequalities as in \eqref{eq:coercivity}, here reformulated as an estimate for test functions
\begin{equation}\label{eq:pCoercive}
  \Norm{ADu}{p}\gtrsim\Norm{Du}{p},\qquad\forall u\in\mathscr{D}(\R^n;\C^N)
\end{equation}
for some fixed $p\in(1,\infty)$. Here it is convenient to consider a slightly more general situation, where the operators $A$ and $D=-i\sum_{j=1}^n\hat{D}_j\partial_j$ can change dimensions, so that $\hat{D}_j\in\bddlin(\C^N,\C^M)$ and $A\in L^{\infty}(\R^n;\bddlin(\C^M,\C^K))$. Moreover, the following results work for any $D$ of this form; only in the last one do we impose the additional requirement that
\begin{equation}\label{eq:constDim}
  \dim\range(\hat{D}(\xi))= \text{constant} =:r\qquad\forall\xi\in\R^n\setminus\{0\}.
\end{equation}

This is obviously satisfied by the fundamental operator $D=\nabla\otimes$ for which $\hat{D}(\xi)=i\xi\otimes\in\bddlin(\C^N,\C^{nN})$ has range $\xi\otimes\C^N$ of fixed dimension $N$. The condition~\eqref{eq:constDim} also holds for all the operators $D$ considered earlier in the paper, i.e., under the assumptions \eqref{eq:D1} and \eqref{eq:D2} made in Section~\ref{sec:setup}. In fact, the consequent condition \eqref{eq:equivCondD} implies that the projection of $\C^N$ onto $\range(\hat{D}(\xi))$ along $\kernel(\hat{D}(\xi))$ is given by
\begin{equation*}
  \proj_{\hat{D}(\xi)}^1=\int_{\Gamma(\xi)}(\lambda-\hat{D}(\xi))^{-1}\ud\lambda,\quad
  \Gamma(\xi)=\partial[S_{\omega'}\cap A(\frac12\kappa\abs{\xi},2M\abs{\xi})],
\end{equation*}
which depends continuously on $\xi\in\R^n\setminus\{0\}$. Now \eqref{eq:constDim} follows easily by a compactness argument from the observation that $\range(\hat{D}(\xi))$ only depends on $\xi^0=\abs{\xi}^{-1}\xi$ and a simple fact about projections:

\begin{lemma}
If two finite-dimensional projections satisfy $\Norm{P-P'}{}<1$, then their ranges have equal dimension.
\end{lemma}

\begin{proof}
Let $u\in\range(P)$. Then $P'u=[P+(P'-P)]u=[I+(P'-P)]u$, and $I+(P'-P)$ is invertible. Thus $P':\range(P)\to\range(P')$ is injective, hence $\dim\range(P')\geq\dim\range(P)$. The claim follows by symmetry.
\end{proof}

We now turn to conditions on the symbol $\hat{D}(\xi)$ related to \eqref{eq:pCoercive}.

\begin{proposition}\label{prop:necCond}
Suppose that the coercivity estimate \eqref{eq:pCoercive} holds for some $p\in[1,\infty)$. Then
\begin{equation}\label{eq:necCond}
  \abs{A(x)\hat{D}(\xi)v}\gtrsim\abs{\hat{D}(\xi)v},\qquad
  \forall \xi\in\R^n,\quad\forall v\in\C^N,\quad\text{a.e. }x\in\R^n.
\end{equation}
\end{proposition}

\begin{proof}
If $u(x)=\eps\psi(x)e^{ix\cdot\xi/\eps}v$, where $\psi\in\mathscr{D}(\R^n)$, $v\in\C^N$, then
\begin{equation*}
   Du(x)=\sum_{j=1}^n\hat{D}_jv\big(i\xi_j\psi(x)+\eps\partial_j\psi(x)\big)e^{ix\cdot\xi/\eps},
\end{equation*}
and
\begin{equation*}
  \Norm{ADu}{p}=\BNorm{A\sum_{j=1}^n\hat{D}_jv\big(i\xi_j\psi+\eps\partial_j\psi\big)}{p}
  \underset{\eps\to 0}{\longrightarrow}\BNorm{A\sum_{j=1}^n\hat{D}_j\xi_jv\psi}{p}
  =\Norm{A\hat{D}(\xi)v\psi}{p}.
\end{equation*}
Let $\psi(x)=\delta^{-n/p}\phi(\delta^{-1}(x-x_0))$, where $0\leq\phi\in\mathscr{D}(\R^n)$ with $\int\phi^p=1$, and $x_0$ be a Lebesgue point  of $A$. Then
\begin{equation*}
  \Norm{A\hat{D}(\xi)v\psi}{p}
  =\Big(\int_{\R^n}\abs{A(x)\hat{D}(\xi)v}^p\phi^p\big(\frac{x-x_0}{\delta}\big)\frac{\ud x}{\delta^n}\Big)^{1/p}
  \underset{\delta\to 0}{\longrightarrow}\abs{A(x_0)\hat{D}(\xi)v}.
\end{equation*}
Since the same reasoning holds with the identity $I$ in place of $A$, the conclusion follows for all $x\in\R^n$, which are Lebesgue points of $A\in L^{\infty}(\R^n;\bddlin(\C^M,\C^K))$.
\end{proof}

\begin{proposition}\label{prop:suffCond}
Suppose that the pointwise coercivity condition
\begin{equation}\label{eq:suffCond}
  \Babs{A(x)\sum_{j=1}^n\hat{D}_j v_j}\gtrsim\Babs{\sum_{j=1}^n\hat{D}_j v_j},\qquad
  \forall v_1,\ldots,v_n\in\C^N,\quad\text{a.e. }x\in\R^n.
\end{equation}
is satisfied. Then for all $p\in[1,\infty)$,  the coercivity estimate \eqref{eq:pCoercive} holds.
\end{proposition}

\begin{proof}
It suffices to observe that $\abs{A(x)Du(x)}\gtrsim\abs{Du(x)}$ by the assumption applied to $v_j=\partial_j u(x)$, take the $p$th power, and integrate over $x\in\R^n$.
\end{proof}

The simple sufficient condition \eqref{eq:suffCond} is not necessary, as we will see after showing that the weaker necessary condition \eqref{eq:necCond} (which corresponds to vectors $v_j$ of the special form $v_j=\xi_j v$ in \eqref{eq:suffCond}) is also sufficient in the following situation:

\begin{proposition}\label{prop:constCoercive}
Suppose that $D$ satisfies the condition \eqref{eq:constDim} and that $A\in\bddlin(\C^M,\C^K)$ is a constant matrix such that
\begin{equation}\label{eq:constCond}
  \abs{A\hat{D}(\xi)v}\gtrsim\abs{\hat{D}(\xi)v},\qquad
  \forall \xi\in\R^n,\quad\forall v\in\C^N,
\end{equation}
Then for all $p\in(1,\infty)$,  the coercivity estimate \eqref{eq:pCoercive} holds.
\end{proposition}

\begin{proof}
Let $E(\xi)$ be the orthogonal projection of $\C^M$ onto $\range(\hat{D}(\xi))=\range(\hat{D}(\xi)\hat{D}(\xi)^*)$. Since this space depends on $\xi\neq 0$ only through the angular component $\xi^0=\abs{\xi}^{-1}\xi$, so does $E(\xi)$. For a positive (in the sense of self-adjoint operators) matrix $A$, the $k$th largest eigenvalue, which coincides with the $k$th singular value, depends continuously on $A$ with respect to the operator norm, and hence the $r$th largest (and thus, by \eqref{eq:constDim}, the smallest positive) eigenvalue $\lambda_r(\xi)$ of $\hat{D}(\xi)\hat{D}(\xi)^*$ depends continuously on $\xi$. Since $\lambda_r(\xi)$ is separated from $0$ for a fixed $\xi\neq 0$, it follows by compactness that $\lambda_r(\xi)\geq 2\delta>0$ for all $\xi$ in a neighbourhood of the unit sphere. This in turn implies that
\begin{equation*}
  E(\xi)=I-\frac{1}{2\pi i}\oint_{\abs{\lambda}=\delta}(\lambda-\hat{D}(\xi)\hat{D}(\xi)^*)^{-1}\ud\lambda
\end{equation*}
for all these $\xi$, and it follows that $E\in C^{\infty}(\R^n\setminus\{0\};\bddlin(\C^M))$. Similarly, the orthogonal projection $F(\xi)$ of $\C^K$ onto $\range(A\hat{D}(\xi))$ defines a function $F\in C^{\infty}(\R^n\setminus\{0\};\bddlin(\C^K))$.

We then define a linear operator $M(\xi):\C^K\to\C^M$ separately on $\range(F(\xi))=A\range(E(\xi))$ and $\range(I-F(\xi))$ as follows:
\begin{equation*}
  M(\xi)AE(\xi):=E(\xi),\qquad M(\xi)(I-F(\xi)):=0.
\end{equation*}
Remark~\ref{rem:justify} below shows that this pointwise definition makes $M$ into a continuous function on $\R^n\setminus\{0\}$.
We can then differentiate these defining equalities, a priori in the sense of distributions. Note that $E$ and $F$ are already known to be smooth, so that their product with a distribution is well-defined. Taking the derivative of order $\alpha\in\N^n$ and moving some terms to the other side, we get
\begin{align*}
  (\partial^{\alpha}M)AE &=E-\sum_{0\neq\beta\leq\alpha}\binom{\alpha}{\beta}(\partial^{\alpha-\beta}M)A\partial^{\beta}E,\\
  (\partial^{\alpha}M)(I-F) &=\sum_{0\neq\beta\leq\alpha}\binom{\alpha}{\beta}(\partial^{\alpha-\beta}M)\partial^{\beta}F.
\end{align*}
Assuming that we already know that all the derivatives of $M$ of order strictly less than $\alpha$ coincide with continuous functions, say, in a neighbourhood of the unit sphere, we obtain the same conclusion for $M$; see again Remark~\ref{rem:justify} for details.

By induction and homogeneity, we have $M\in C^{\infty}(\R^n\setminus\{0\};\bddlin(\C^K,\C^M))$, and such a function automatically satisfies Mihlin's multiplier conditions.
Thus Mihlin's multiplier theorem provides the desired estimate
\begin{equation*}
  \Norm{Du}{p}=\Norm{[M\widehat{ADu}]^{\vee}}{p}\lesssim\Norm{ADu}{p}.\qedhere
\end{equation*}
\end{proof}

\begin{remark}
It was not essential for the argument that the projections $E$ and $F$ are orthogonal, only that they satisfy Mihlin's multiplier estimates.  In the special case when $\hat{D}(\xi)$ and $A\hat{D}(\xi)$ are both bisectorial with $\abs{\hat{D}(\xi)w} \gtrsim \abs{\xi}\abs{w}$ for $w\in
\range(\hat{D}(\xi))$, we could take $E(\xi)$ and $F(\xi)$ as the corresponding spectral projections.
\end{remark}

\begin{remark}\label{rem:justify}
Let $M$ be a function defined pointwise by conditions of the form
\begin{equation*}
  M(\xi)AE(\xi)=G(\xi),\qquad M(\xi)(I-F(\xi))=H(\xi),
\end{equation*}
where $A$, $E$ and $F$ are as in the previous proof, and $G$ and $H$ are continuous functions (say, of $\xi\neq 0$). Then $\abs{M(\xi)AE(\xi)v}=\abs{G(\xi)E(\xi)v}\lesssim\abs{E(\xi)v}\lesssim\abs{AE(\xi)v}$ so $M$ is pointwise well-defined, and $M$ is also continuous. This justifies, first of all, treating the pointwise-defined $M$ in the previous proof as a distribution, and second, identifying its derivatives as continuous functions.

Let us prove this claim. We fix some $\xi\neq 0$, and consider the difference of $M(\xi)$ and $M(\xi')$ for a near-by point $\xi'$. Let $v$ be a fixed vector. By definition, we have that $F(\xi)v=AE(\xi)w$ for some vector $w$, which we may choose to be from $\range(E(\xi))$. Then $\abs{w}=\abs{E(\xi)w}\lesssim\abs{AE(\xi)w}=\abs{F(\xi)v}\leq\abs{v}$. From now on, let us write $M:=M(\xi)$ and $M':=M(\xi')$, with a similar convention for the other relevant functions. We have
\begin{align*}
  M'v
  &=M'Fv+M'(I-F)v=M'AEw+[M'(I-F')v+M'(F'-F)v] \\
  &=[M'AE'w+M'A(E-E')w]+H'v+M'(F'-F)v \\
  &=G'w+M'A(E-E')w+[Hv+(H'-H)v]+M'(F'-F)v,
\end{align*}
and here $Hv=M(I-F)v$ and
\begin{equation*}
  G'w=Gw+(G'-G)w=MAEw+(G'-G)w=MFv+(G'-G)w.
\end{equation*}
Since $Mv=MFv+M(I-F)v$, it follows that
\begin{equation*}
  M'v-Mv=(G'-G)w+M'A(E-E')w+(H'-H)v+M'(F'-F)v,
\end{equation*}
and all summands on the right contain bounded factors multiplied by a difference of a continuous function at $\xi$ and $\xi'$; hence the continuity of $M$ follows.
\end{remark}

\begin{example}[Korn's inequality]\label{ex:Korn}
We illustrate the sufficient condition \eqref{eq:constCond}, and show the non-necessity of \eqref{eq:suffCond}, by deducing the following well-known Korn's inequality from Proposition~\ref{prop:constCoercive}: For $u=(u_i)_{i=1}^n\in\mathscr{D}(\R^n;\C^n)$, there holds
\begin{equation*}
  \sum_{i,j=1}^n\Norm{\partial_i u_j}{p}\lesssim\sum_{i,j=1}^n\Norm{\partial_i u_j+\partial_j u_i}{p}.
\end{equation*}
In fact, this can be written as
\begin{equation*}
  \Norm{Du}{p}\lesssim\Norm{ADu}{p},
\end{equation*}
where $D=\nabla\otimes$ satisfies \eqref{eq:constDim} as already pointed out, and $A:\in\bddlin(\C^n\otimes\C^n)$ is the symmetrizer defined by $(Aw)_{ij}:=w_{ij}+w_{ji}$. The symbolic condition \eqref{eq:constCond} follows at once from
\begin{align*}
  \abs{A\hat{D}(\xi)v}^2 &=\sum_{i,j=1}^n\abs{\xi_i v_j+\xi_j v_i}^2
  =\sum_{i,j=1}^n\big(\abs{\xi_i v_j}^2+\abs{\xi_j v_i}^2+2\xi_i\xi_j\Re(v_j\bar{v}_i)\big) \\
  &=2\abs{\xi}^2\abs{v}^2+2\Re(\xi\cdot\bar{v})(\xi\cdot v)
    =2\big(\abs{\xi}^2\abs{v}^2+\abs{\xi\cdot v}^2\big) \\
  &\geq 2\abs{\xi}^2\abs{v}^2=2\abs{\hat{D}(\xi)v}^2,\qquad\qquad\forall v\in\C^n,
\end{align*}
and hence Korn's inequality is indeed a consequence of Proposition~\ref{prop:constCoercive}.

However, the vectors $\sum_{j=1}^n\hat{D}_jv_j=\sum_{j=1}^n e_j\otimes v_j$ appearing in Proposition~\ref{prop:suffCond} now cover all of $\C^n\otimes\C^n$ as $v_1,\ldots,v_n\in\C^n$. Thus condition \eqref{eq:suffCond} asks for the boundedness from below of $A$ on all of $\C^n\otimes\C^n$, and this clearly cannot hold, since $A$ annihilates all the antisymmetric vectors $(v_{ij})_{i,j=1}^n$ with $v_{ji}=-v_{ij}$.
\end{example}

\subsection*{Acknowledgments}
Much of this research took place during Hyt\"onen's visit to the Centre for Mathematics and its Applications at the Australian National University. He would like to thank the CMA for the hospitality, the financial support and the kind invitation to participate in the 2009 Special Year in Spectral Theory and Operator Theory. Hyt\"onen also gratefully acknowledges support from the Academy of Finland, projects 114374, 130166 and 133264. McIntosh was supported by the CMA and by the Australian Government through the Australian Research Council. Together we thank Pierre Portal for our extensive joint collaboration on the $L^p$ functional calculus of Hodge--Dirac operators and first order systems, as well as Sergey Ajiev, Pascal Auscher and Chema Martell for informative discussions on the topic of this paper.  

\def\cprime{$'$}

\end{document}